\documentclass[12pt]{article}

\usepackage{amsmath}
\usepackage{amsthm}
\usepackage{amssymb}
\usepackage{amsfonts}
\usepackage{bbm}
\usepackage{dsfont}
\usepackage{stmaryrd}
\usepackage{graphicx}
\usepackage{hyperref}

\newtheorem{theo}{Theorem}
\newenvironment{theorem}{\vspace{4mm}\begin{theo}}{\end{theo}}
\newtheorem{lem}[theo]{Lemma}
\newenvironment{lemma}{\vspace{4mm}\begin{lem}}{\end{lem}}
\newtheorem{coro}[theo]{Corollary}
\newenvironment{corollary}{\vspace{4mm}\begin{coro}}{\end{coro}}
\newtheorem{rem}[theo]{Remark}
\newenvironment{remark}{\vspace{4mm}\begin{rem}\rm}{\end{rem}}
\newtheorem{conjec}{Conjecture}

\newtheorem{prop}[theo]{Proposition}

\newtheorem{defi}{Definition}

\newtheoremstyle{citing}{}{}{\itshape}{}{\bfseries}{.}%
 { }{\thmnote{#3}}
\theoremstyle{citing}

\newcommand{\Z}{\mathbb{Z}}
\newcommand{\N}{\mathbb{N}}
\newcommand{\R}{\mathbb{R}}
\newcommand{\ind}{\scalebox{1.2}{\raisebox{-0.2mm}{$\mathds{1}$}}}
\newcommand{\abs}[1]{\left\vert #1 \right\vert}	
\newcommand{\norm}[1]{\left\Vert #1 \right\Vert}	
\renewcommand{\l}{\langle}
\renewcommand{\r}{\rangle}
\renewcommand{\a}{\alpha}
\newcommand{\g}{\gamma}

\newcommand{\eps}{\varepsilon}
\renewcommand{\O}{\Omega}
\newcommand{\T}{\mathcal{T}}
\newcommand{\nlongleftrightarrow}{\longarrownot\longleftrightarrow}

\begin{document}

\title{Rapid mixing of Swendsen-Wang and single-bond dynamics in two dimensions}
\author{Mario Ullrich\footnote{The author was supported by the DFG GK 1523.}\\
	Mathematisches Institut, Universit\"at Jena\\   
	email: mario.ullrich@uni-jena.de}



\maketitle

\begin{abstract}
We prove that the spectral gap of the Swendsen-Wang dynamics for 
the random-cluster model on arbitrary graphs with $m$ edges is 
bounded above by $16 m \log m$ times the spectral gap of the 
single-bond (or heat-bath) dynamics. 
This and the corresponding lower bound (from \cite{U3}) imply 
that rapid mixing of these two dynamics is equivalent. \\
Using the known lower bound on the spectral gap of the Swendsen-Wang 
dynamics for the two dimensional square lattice $\Z_L^2$ of side 
length $L$ at high temperatures and a result for the single-bond dynamics on 
dual graphs, we obtain rapid mixing of both dynamics on $\Z_L^2$ at all 
non-critical temperatures.
In particular this implies, as far as we know, the first proof of rapid mixing 
of a classical Markov chain for the Ising model on $\Z_L^2$ at all temperatures.
\end{abstract}

\section{Introduction}

Markov chains for the random-cluster model and the closely related 
$q$-state Potts model are the topic of many research articles from various 
areas of mathematics and statistical physics. 
The probably most studied model is the Ising model (or 2-state Potts model) 
on the two-dimensional square lattice. While there is almost complete 
knowledge about the mixing properties of single-spin dynamics, such as the 
heat-bath dynamics, there are only a few results on cluster algorithms, 
such as Swendsen-Wang, or dynamics on the corresponding random-cluster 
model. The single-spin dynamics on $\Z_L^2$, i.e. the two dimensional 
square lattice of side length $L$, is known to mix rapidly 
above the critical temperature \cite{MO1} and 
below the critical temperature the mixing time is exponential in the side 
length $L$, see \cite{CGMS}. See also \cite{M} for an excellent 
survey of the (at that time) known results.
Only recently it was proven by Lubetzky and Sly \cite{LS}
that the single-spin dynamics is also rapidly mixing at the critical 
temperature. One approach to overcome the \emph{torpid} (or slow) 
mixing at low temperatures was to consider cluster algorithms that change 
the spin of a large portion of vertices at once. The most successful approach 
(so far) is the Swendsen-Wang dynamics (SW) that is based on the close relation 
between Potts and random-cluster models, see \cite{SW} and \cite{ES}.
But although it is conjectured that this dynamics is rapidly mixing at high and 
at low temperatures, again most of the results concern high temperatures. 
Known results for general graphs include rapid mixing on trees and complete 
graphs at all temperatures, see e.g. \cite{CF}, \cite{CDFR} and \cite{LNP}, 
and graphs with small maximum degree at high temperatures , see \cite{CF} and 
\cite{Hu}. Additionally, it is proven that, for bounded degree graphs, 
rapid mixing of single-spin dynamics implies rapid mixing of SW \cite{U2}. 
At low temperatures there are only a few results on the mixing time of 
SW. Beside the results for trees and complete graphs, we are only aware 
of two articles concerning the low temperature case, \cite{M2} and \cite{Hu}. 
While Huber \cite{Hu} states rapid mixing for temperatures below some constant 
that depends on the size of the graph, Martinelli \cite{M2} gave a result 
for hypercubic subsets of $\Z^d$ at sufficiently low temperatures that do not 
depend on the side length. Additionally to the rapid mixing results for SW 
there are some results on torpid mixing. These include torpid mixing for 
the $q$-state Potts model at the critical temperature on the complete graph 
for all $q\ge3$ \cite{GJ} and on hypercubic subsets of $\Z^d$ for $q$ 
sufficiently large \cite{BCT}.

In this article we study the mixing properties of the Swendsen-Wang and the 
heat-bath dynamics for the random-cluster model 
(In fact, SW can be seen as a Markov chain for random-cluster and Potts 
models.) and we prove that the spectral gap of SW is bounded above by  
some polynomial in the size of the graph times the spectral gap of the 
heat-bath dynamics. In particular, this implies rapid mixing of SW for the 
Potts model on the two-dimensional square lattice at all temperatures below 
the critical one.

To state our results in detail, we first have to define the models and the 
algorithms. Let $G=(V,E)$ be a graph with finite vertex set $V$ and 
edge set $E$. 
The \emph{random-cluster model} 
(also known as the FK-model) with parameters $p\in(0,1)$ and $q\in\N$, 
see Fortuin and Kasteleyn \cite{FK}, is defined on the graph $G$ by its 
state space $\O_{\rm RC}=\{A: A\subseteq E\}$ and the RC measure
\[
\mu(A) \;:=\; \mu^G_{p,q}(A) \;=\; 
\frac1{Z(G,p,q)}\,
\left(\frac{p}{1-p}\right)^{\abs{A}}\,q^{c(A)},
\]
where $c(A)$ is the number of connected components 
in the graph $(V,A)$, counting isolated vertices as a component, 
and $Z$ is the normalization constant that makes $\mu$ a 
probability measure. Note that this model is well-defined also for 
non-integer values of $q$, but we do not need this generalization here. 
See \cite{G1} for further details and related topics.

A closely related model is 
the \emph{$q$-state Potts model} on $G$ at inverse temperature 
$\beta\ge0$, that is defined as the set of 
possible \emph{configurations} 
$\O_{\rm P}=[q]^V$, where $[q]\,{:=}\,\{1,\dots,q\}$ 
is the set of \emph{colors} (or spins),
together with the probability measure
\vspace{1mm}
\[
\pi(\sigma) \;:=\; \pi^{G}_{\beta,q}(\sigma) \;=\; 
	\frac1{Z(G,1-e^{-\beta},q)}\,
	\exp\left\{\beta\,\sum_{u,v:\, \{u,v\}\in E}
	\ind\bigl(\sigma(u)=\sigma(v)\bigr)\right\}
\]
for $\sigma\in\O_{\rm P}$, where 
$Z(\cdot,\cdot,\cdot)$ is the same normalization 
constant as for the RC model (see \cite[Th. 1.10]{G1}). 
For $q=2$ this model is called \emph{Ising model}.

The connection of these models is given by a coupling of the RC 
and the Potts measure in the case $p=1-e^{-\beta}$. 
Let $\sigma\in\O_{\rm P}$ and $A\in\O_{\rm RC}$. 
Then the joint measure of 
$(\sigma,A)\in\O_{\rm J}:=\O_{\rm P}\times\O_{\rm RC}$ is defined by
\[
\nu(\sigma,A)\;:=\;\nu_{p,q}^G(\sigma,A)
	\;=\;\frac{1}{Z(G,p,q)}\,
	\left(\frac{p}{1-p}\right)^{\abs{A}}\,\ind(A\subset E(\sigma)),
\]
where
\[
E(\sigma) \;:=\; \Bigl\{\{u,v\}\in E:\,\sigma(u)=\sigma(v)\Bigr\}.
\]
The marginal distributions of $\nu$ are exactly $\pi$ and 
$\mu$, respectively, and we will call $\nu$ the \emph{FKES} 
(Fortuin-Kasteleyn-Edwards-Sokal) \emph{measure}, see e.g. 
\cite{ES} and \cite{G1}.

The \emph{Swendsen-Wang dynamics} (SW) uses this coupling implicitly 
in the following way. Suppose the SW at time $t$ is in the state 
$A_t\in\O_{\rm RC}$. We choose $\sigma_t\in\O_{\rm P}$ with respect 
to the measure $\nu(\cdot,A_t)$, i.e. every connected component of 
$(V,A_t)$ is colored independently and uniformly at random with a 
color from $[q]$. Then take $E(\sigma_t)$ and delete each edge 
independently with probability $1-p$ to obtain $A_{t+1}\in\O_{\rm RC}$, 
which can be seen as sampling from $\nu(\sigma_t,\cdot)$. 
Denote by $P_{\rm SW}$ the transition matrix of this Markov chain. 
Of course, we can make these two steps in reverse order to obtain 
a Markov chain for the $q$-state Potts model with transition matrix 
$\widetilde{P}_{\rm SW}$.

The \emph{heat-bath dynamics} (HB) for the random-cluster model is a local 
Markov chain that, given the current state $A_t\in\O_{\rm RC}$, 
sets $A_{t+1}=A_t$ with probability $\frac12$ and otherwise 
chooses a edge $e\in E$ uniformly at random and changes the state at 
most at the edge $e$ with respect to the conditional measure given all 
the other edges, which is sampling of $A_{t+1}$ 
from the conditional measure $\mu(\cdot\mid\{A_t\cup e, A_t\setminus e\})$. 
The transition matrix of this chain is denoted $P_{\rm HB}$.

The \emph{spectral gap} of a Markov chain with transition matrix $P$ 
is defined by 
\[
\lambda(P) \;:=\; 1 - \max\Bigl\{\abs{\xi}:\, \xi 
		\text{ is an eigenvalue of } P,\; \xi\neq1\Bigr\}.
\]

We prove

\begin{theorem} \label{th:main}
Let $P_{\rm SW}$ (resp. $P_{\rm HB}$) be the transition matrix of the 
Swendsen-Wang (resp. heat-bath) dynamics for the random-cluster model 
on a graph with $m$ edges. Then 
\[
\lambda(P_{\rm SW}) \;\le\; 16m \log(m)\; \lambda(P_{\rm HB}).
\]
\end{theorem}

Using the corresponding lower bound, which was proven in \cite{U3} 
 we obtain that SW is rapidly mixing 
if and only if HB is rapidly mixing, since the spectral gaps can 
differ only by a polynomial in the number of edges of the graph.
Furthermore we prove that the heat-bath dynamics for the RC model 
on a \emph{planar} graph $G$ with parameters $p$ and $q$ has the same spectral 
gap than the heat-bath dynamics for the \emph{dual model}, which is the 
random-cluster model on the dual graph $G^\dag$ 
(see Section \ref{subsec:dual} for definitions)
with parameters $p^*$ and $q$, where $p^*$ satisfies 
$\frac{p^*}{1-p^*}=\frac{q(1-p)}{p}$. We denote the dynamics for the dual 
model by $P^\dag_{\rm HB}$ (resp. $P^\dag_{\rm SW}$).
This was probably known before, but we could not found a reference.
It follows

\begin{corollary} \label{coro:SW_dual}
Let $P_{\rm SW}$ be the transition matrix of the 
Swendsen-Wang dynamics for the random-cluster model 
on a planar graph $G$ with $m$ edges and let $P^\dag_{\rm SW}$ be the 
SW dynamics for the dual model.
Then there exists a constant 
$c\le16\min\{q,\frac1{1-p}\}$, such that
\[
\lambda(P_{\rm SW}) \;\le\; c m \log(m)\; \lambda(P^\dag_{\rm SW}).
\]
\end{corollary}

If we consider the two-dimensional square lattice of side length $L$, 
i.e. the graph $\Z_L^2=(V_L,E_L)$ with $V_L=[L]^2$ and 
$E_L=\bigl\{\{u,v\}\in\binom{V_L}{2}:\, \abs{u_1-v_1}+\abs{u_2-v_2}=1 \bigr\}$, 
we can deduce the following from the results of \cite{U2}.

\begin{theorem} \label{th:SW_square}
Let $P_{\rm SW}$ be the transition matrix of the 
Swendsen-Wang dynamics for the random-cluster model 
on $\Z^2_L$ with parameters $p$ and $q$. Let $N=L^2$. 
Then there exist constants $c_p,c',C<\infty$ such that 
\begin{itemize}
\item\quad $\lambda(P_{\rm SW})^{-1} \;\le\; c_p N$ 
			\qquad\qquad for $p < p_c(q)$,
\item\quad $\lambda(P_{\rm SW})^{-1} \;\le\; c_p N^2 \log N$ 
			\quad for $p > p_c(q)$,
\item\quad $\lambda(P_{\rm SW})^{-1} \;\le\; c' N^C$ 
			\qquad\quad\, for $q=2$ and $p = p_c(2)$,
\end{itemize}
where $p_c(q)\,=\,\frac{\sqrt{q}}{1+\sqrt{q}}$.
\end{theorem}

An immediate consequence is the following corollary.

\begin{corollary} \label{coro:SW_square}
Let $\widetilde{P}_{\rm SW}$ be the transition matrix of the 
Swendsen-Wang dynamics for the $q$-state Potts model 
on $\Z^2_L$ at inverse temperature $\beta$. Let $N=L^2$. 
Then there exist constants $c_\beta,c',C<\infty$ such that 
\begin{itemize}
\item\quad $\lambda(\widetilde{P}_{\rm SW})^{-1} \;\le\; c_\beta N$ 
			\qquad\qquad for $\beta < \beta_c(q)$,
\item\quad $\lambda(\widetilde{P}_{\rm SW})^{-1} \;\le\; c_\beta N^2 \log N$ 
			\quad for $\beta > \beta_c(q)$,
\item\quad $\lambda(\widetilde{P}_{\rm SW})^{-1} \;\le\; c' N^C$ 
			\qquad\quad\, for $q=2$ and $\beta = \beta_c(2)$,
\end{itemize}
where $\beta_c(q)\,=\,\log(1+\sqrt{q})$.
\end{corollary}

This seems to be the first prove of rapid mixing of a classical 
Markov chain for the Ising model at all temperatures. 
In fact, in \cite{U2} a somehow artificial Markov chain, that 
makes a additional step at the dual graph, is proven to be rapid.

We also obtain for the heat-bath dynamics 

\begin{theorem} \label{th:HB_square}
Let $P_{\rm HB}$ be the transition matrix of the 
heat-bath dynamics for the random-cluster model 
on $\Z^2_L$ with parameters $p$ and $q$. Let $N=L^2$. 
Then there exist constants $c_p,c',C<\infty$ such that 
\begin{itemize}
\item\quad $\lambda(P_{\rm HB})^{-1} \;\le\; c_p N^2 \log N$ 
			\quad for $p \neq p_c(q)$,
\item\quad $\lambda(P_{\rm HB})^{-1} \;\le\; c' N^C$ 
			\qquad\quad\, for $q=2$ and $p = p_c(2)$,
\end{itemize}
where $p_c(q)\,=\,\frac{\sqrt{q}}{1+\sqrt{q}}$.
\end{theorem}

The results of \cite{U2}, and hence the proofs of Theorems \ref{th:SW_square} 
and \ref{th:HB_square}, 
rely ultimately on the rapid mixing results for the heat-bath dynamics 
for the Potts (resp. Ising) model that were proven over the last decades. 
To state only some of them, see e.g. \cite{LS}, \cite{MO1}, \cite{MOS}l and 
\cite{A} together with the proof of exponential decay of connectivities 
up to the critical temperature from \cite{BDC}.
These articles give an almost complete picture over what is known so far 
about mixing of single-spin dynamics in $\Z^2_L$.

The plan of this article is as follows. 
In Section~\ref{sec:gap}, we introduce the necessary notation related 
to the spectral gap of Markov chains. 
Section~\ref{sec:alg} contains a more detailed description of the algorithms 
and the definition of the ``building blocks'' that are necessary to 
represent the dynamics on the FKES model.
In Section~\ref{sec:proof_main} we will prove Theorem~\ref{th:main}, and 
in Section~\ref{sec:proof_square} we introduce the notion of dual graphs and 
prove the remaining results from above. 

\section{Spectral gap and mixing time} \label{sec:gap}

As stated in the introduction, we want to estimate the efficiency of  
Markov chains. For an introduction to Markov 
chains and techniques to bound the convergence rate to the stationary 
distribution, see e.g. \cite{LPW}. 
In this article we consider the spectral gap as measure of the efficiency. 
Let $P$ be the transition matrix of a Markov chain with state space 
$\Omega$ that is ergodic, i.e. irreducible and aperiodic, 
and has unique stationary measure $\pi$. 
Additionally let the Markov chain $P$ be \emph{reversible} with 
respect to $\pi$, i.e.
\[
\pi(x)\,P(x,y) \;=\; \pi(y)\,P(y,x) \quad \text{ for all } x,y\in\O.
\]

Then we know that 
the spectral gap of the Markov chain can be expressed in terms of norms 
of the (Markov) operator $P$ that maps from 
$L_2(\pi):=(\R^\O,\pi)$ to $L_2(\pi)$, where 
inner product and norm are given by 
$\l f,g\r_\pi=\sum_{x\in\O}f(x) g(x) \pi(x)$ and 
$\Vert f\Vert_\pi^2:=\sum_{x\in\O}f(x)^2\pi(x)$, respectively. 
The operator is defined by
\begin{equation} \label{eq:map}
Pf(x) \;:=\; \sum_{y\in\O}\,P(x,y)\,f(y)
\end{equation}
and represents the expected value of the function $f$ after one step of 
the Markov chain starting in $x\in\O$. 
The \emph{operator norm} of $P$ is
\[
\Vert P\Vert_\pi \;:=\; \Vert P\Vert_{L_2(\pi)\to L_2(\pi)} 
\;=\; \max_{\Vert f\Vert_\pi\le1} \Vert Pf\Vert_\pi
\]
and we use $\Vert\cdot\Vert_\pi$ interchangeably for functions and 
operators, because it will be clear from the context which 
norm is used. It is well known that 
$\lambda(P)=1-\norm{P-S_\pi}_\pi$ for reversible $P$, 
where $S_\pi(x,y)=\pi(y)$, and that reversibility of $P$ is equivalent 
to self-adjointness of the corresponding Markov operator, i.e. 
$P=P^*$, where $P^*$ is the (\emph{adjoint}) operator that 
satisfies 
$\l f, Pg\r_{\pi} \;=\; \l P^*f, g\r_{\pi}$ for all 
$f,g\in L_2(\pi)$. The transition matrix that corresponds to the 
adjoint operator satisfies 
\[
P^*(x,y) = \frac{\pi(y)}{\pi(x)} P(y,x).
\]

If we are considering a \emph{family} of state spaces
$\{\O_n\}_{n\in\N}$ with a corresponding family of Markov chains 
$\{P_n\}_{n\in\N}$, we say that the chain is \emph{rapidly mixing}
for the given family if $\lambda(P_n)^{-1} = \mathcal{O}(\log(|\O_n|)^C)$
for all $n\in\N$ and some $C < \infty$.

In several (or probably most of the) articles on mixing properties of 
Markov chains the authors prefer to use the \emph{mixing time} as 
measure of efficiency, which is defined by
\[
\tau(P) \;:=\; \min\left\{t:\; 
	\max_{x\in\O}\sum_{y\in\O}\abs{P^t(x,y)-\pi(y)}\,\le\,\frac{1}{e}\right\}.
\]

The mixing time and spectral gap of a Markov chain (on finite state spaces) 
are closely related by the following inequality, see e.g. 
\cite[Theorem~12.3 \& 12.4]{LPW}.

\begin{lemma} \label{lemma:mixing-gap}
Let $P$ be the transition matrix of a reversible, ergodic Markov chain with 
state space $\O$ and stationary distribution $\pi$. Then
\[
\lambda(P)^{-1}-1 \;\le\; \tau(P) 
\;\le\; \log\left(\frac{2e}{\pi_{\rm min}}\right)\,\lambda(P)^{-1}, 
\]
where $\pi_{\rm min}:=\min_{x\in\O}\pi(x)$.
\end{lemma}

In particular, we obtain the following for the random-cluster model. 

\begin{corollary} \label{coro:mixing-gap}
Let $P$ be the transition matrix of a reversible, ergodic Markov chain 
for the random-cluster model on $G=(V,E)$ with parameters $p$ and $q$. 
Then
\[
\lambda(P)^{-1}-1 \;\le\; \tau(P) 
\;\le\; \left(2 + \abs{E} \log\frac{1}{p(1-p)} + \abs{V} \log q\right)\,
			\lambda(P)^{-1}. 
\]
\end{corollary}

Therefore, all results of this article can also be written in terms of 
the mixing time, loosing the same factor as in Corollary~\ref{coro:mixing-gap}.

\section{Joint representation of the algorithms} \label{sec:alg}

In order to make our description of the considered Markov chains complete, 
we state in this section formulas for their transition matrices.
Additionally, we introduce another local Markov chain that will be 
necessary for the further analysis and introduce a representation of the 
dynamics on (joint) FKES model.

The \emph{Swendsen-Wang dynamics} (on the RC model), as stated in the 
introduction, is based on the given connection of the random cluster 
and Potts models and has the transition matrix
\begin{equation} \label{eq:SW}
P_{\rm SW}(A,B) \;=\; q^{-c(A)}\,\left(\frac{p}{1-p}\right)^{\abs{B}}\,
	\sum_{\sigma\in\O_{\rm P}} \,(1-p)^{\abs{E(\sigma)}}\,
		\ind\bigl(A\cup B \subset E(\sigma)\bigr).
\end{equation}
Recall that we denote by $\widetilde{P}_{\rm SW}$ the Swendsen-Wang dynamics 
for the Potts model and note that both dynamics have the same spectral gap, 
see \cite{U2}.

%

The second algorithm we want to analyze is the (lazy) \emph{heat-bath dynamics}. 
Let $A\in\O_{\rm RC}$ be given and denote by $\stackrel{A}{\leftrightarrow}$ 
(resp. $\stackrel{A}{\nleftrightarrow}$) connected (resp. not connected) in 
the subgraph $(V,A)$.
Additionally we use throughout this article $A\cup e$ instead of $A\cup\{e\}$ 
(respectively for $\cap,\setminus$) and denote the endpoints of $e$ by $e^{(1)}$ 
and $e^{(2)}$, i.e. $e=\{e^{(1)},e^{(2)}\}$.
For $A,B\in\O_{\rm RC}$, $A\neq B$, the transition probabilities of the HB 
dynamics are given by

\begin{equation} \label{eq:HB}
P_{\rm HB}(A,B) \;:=\; \frac1{2\abs{E}}\,\sum_{e\in E}\,
\frac{\mu(B)}{\mu(A\cup e)+\mu(A\setminus e)}\;\ind(A\ominus B\subset e),
\end{equation}
where $\ominus$ denotes the symmetric difference and $P_{\rm HB}(A,A)$ is 
chosen such that $P_{\rm HB}$ is stochastic.
Hence, $P_{\rm HB}$ satisfies
\[
P_{\rm HB}(A,B) \;=\;  \frac1{2\abs{E}}\sum_{e\in E}\,\begin{cases}
p, & \text{ for } B=A\cup e \text{ and } 
				e^{(1)}\stackrel{A\setminus e}{\longleftrightarrow}e^{(2)}\\
1-p, & \text{ for } B=A\setminus e \text{ and } 
				e^{(1)}\stackrel{A\setminus e}{\longleftrightarrow}e^{(2)}\\
\frac{p}{p+q(1-p)}, & \text{ for } B=A\cup e \text{ and } 
				e^{(1)}\stackrel{A\setminus e}{\nlongleftrightarrow}e^{(2)}\\
\frac{q(1-p)}{p+q(1-p)}, & \text{ for } B=A\setminus e \text{ and } 
				e^{(1)}\stackrel{A\setminus e}{\nlongleftrightarrow}e^{(2)}
\end{cases}
\]
for $A\neq B$.
The heat-bath dynamics has the advantage that the corresponding HB dynamics 
for the dual model has the same spectral gap, see Section~\ref{subsec:dual}. 
Unfortunately, this Markov chains do not admit a representation on the joint 
model like the SW dynamics. Therefore we introduce the following (non-lazy) 
local dynamics with transition probabilities
\begin{equation} \label{eq:SB}
P_{\rm SB}(A,B) \;=\;  \frac1{\abs{E}}\sum_{e\in E}\,\begin{cases}
p, & \text{ for } B=A\cup e \text{ and } 
				e^{(1)}\stackrel{A}{\longleftrightarrow}e^{(2)}\\
1-p, & \text{ for } B=A\setminus e \text{ and } 
				e^{(1)}\stackrel{A}{\longleftrightarrow}e^{(2)}\\
\frac{p}{q}, & \text{ for } B=A\cup e \text{ and } 
				e^{(1)}\stackrel{A}{\nlongleftrightarrow}e^{(2)}\\
1-\frac{p}{q}, & \text{ for } B=A\setminus e \text{ and } 
				e^{(1)}\stackrel{A}{\nlongleftrightarrow}e^{(2)}.
\end{cases}
\end{equation}
We call this Markov chain the \emph{single-bond dynamics} (SB). 
This chain is inspired by the Swendsen-Wang dynamics since 
$P_{\rm SW}=P_{\rm SB}$ for a graph that consists of two vertices 
connected by a single edge. 
Note that $P_{\rm SW}$, $P_{\rm HB}$ and $P_{\rm SB}$ are reversible 
with respect to $\mu$.

Before we state the representation of $P_{\rm SW}$ and $P_{\rm SB}$ on 
the FKES model, we show that the spectral gaps of $P_{\rm SB}$ and 
$P_{\rm HB}$ are closely related. For this let $I(A,B):=\ind(A=B)$ and 
note that $\frac12(I+P_{\rm SB})$ is the transition matrix of the 
lazy single-bond dynamics.

\begin{lemma} \label{lemma:SB-HB}
For $P_{\rm HB}$ and $P_{\rm SB}$ for the random-cluster model with 
parameters $p$ and $q$ we have
\[
\frac12\,\lambda(P_{\rm SB}) \;\le\; \lambda(P_{\rm HB}) 
	\;\le\; \left(1-p (1-q^{-1})\right)^{-1}\,
						\lambda\left(\frac{I+P_{\rm SB}}{2}\right)
\]
\end{lemma}
\begin{proof}
Using standard comparison ideas, e.g. from \cite[Section~2.A]{DSC2}, 
we obtain that for two transition matrices $P$ and $Q$, 
$P(A,B)\le c Q(A,B)$ for all $A,B\in\O_{\rm RC}$ implies 
$\lambda(P)\le c \lambda(Q)$, where for lazy Markov chains the inequality 
for all $A\neq B$ is sufficient. Additionally we have in general 
$\lambda(P)\le 2 \lambda\bigl(\frac12(I+P)\bigr)$.
Therefore it is enough to prove 
$\frac12(I+P_{\rm SB})(A,B)\le P_{\rm HB}(A,B)
\le (1-p (1-q^{-1}))^{-1}\frac12(I+P_{\rm SB})(A,B)$ for all $A\neq B$, 
which is easy to check.\\
\end{proof}

We want to represent the Swendsen-Wang and the single-bond 
dynamics on the FKES model, which consists of the product state space 
$\O_{\rm J}:=\O_{\rm P}\times\O_{\rm RC}$ and the FKES measure $\nu$. 
This was done first in \cite{U3} and we follow the steps from this article. 
First we introduce the stochastic matrix that defines the mapping (by matrix 
multiplication) from the RC to the FKES model
\begin{equation} \label{eq:M}
M\bigl(B,(\sigma,A)\bigr) \;:=\; q^{-c(B)}\;\ind\bigl(A=B\bigr)\;
	\ind\bigl(B\subset E(\sigma)\bigr).
\end{equation}
Note that $M$ defines an operator (like in \eqref{eq:map}) that maps from 
$L_2(\nu)$ to $L_2(\mu)$ and its adjoint operator $M^*$ can be given 
by the (stochastic) matrix
\[
M^*\bigl((\sigma,A),B\bigr) \;=\; \ind\bigl(A=B\bigr).
\]
The following matrix represents the updates of the RC 
``coordinate'' in the FKES model. 
For $(\sigma,A),(\tau,B)\in\O_{\rm J}$ and $e=\{e^{(1)},e^{(2)}\}\in E$ let
\begin{equation} \label{eq:T-e}
T_e\bigl((\sigma,A),(\tau,B)\bigr) \;:=\; \ind\bigl(\sigma=\tau\bigr)\;
	\begin{cases}
	p, & B=A\cup e  \,\text{ and }\;  \sigma(e^{(1)})=\sigma(e^{(2)})\\
	1-p, & B=A\setminus e  \;\text{ and }\;  \sigma(e^{(1)})=\sigma(e^{(2)})\\
	1, & B=A\setminus e  \;\text{ and }\;  \sigma(e^{(1)})\ne\sigma(e^{(2)}).
	\end{cases}
\end{equation}

The following simple lemma 
shows some interesting properties of the matrices from \eqref{eq:M} and 
\eqref{eq:T-e}, e.g. $\{T_e\}_{e\in E}$ is a family of commuting 
projections in $L_2(\nu)$. This will be important in the proof of 
the main result.

\begin{lemma} \label{lemma:prop}
Let $M$, $M^*$ and $T_e$ be the matrices from above.
Then
\begin{enumerate}
	\renewcommand{\labelenumi}{\rm(\roman{enumi})}
	\item $M^*M$ and $T_e$ are self-adjoint in $L_2(\nu)$.
\vspace{1mm}
	\item $T_e T_e = T_e$ and $T_e T_{e'} = T_{e'} T_e$ for all $e,e'\in E$.
\vspace{1mm}
	\item $\norm{T_{e}}_{\nu}=1$ and $\norm{M^*M}_{\nu}=1$.
\end{enumerate}
\end{lemma}

Now we can state the desired Markov chains with the matrices from above. 

\begin{lemma} \label{lemma:repr}
Let $M$, $M^*$ and $T_e$ be the matrices from above.
Then
\begin{enumerate}
	\renewcommand{\labelenumi}{\rm(\roman{enumi})}
	\item $P_{\rm SW} \,=\, M \left(\prod\limits_{e\in E} T_e\right) M^*$.
\vspace{1mm}
	\item $P_{\rm SB} 
						\,=\, \frac1{\abs{E}}\sum\limits_{e\in E}\,M\,T_e\,M^*
						\,=\, M\left(\frac1{\abs{E}}\sum\limits_{e\in E}\,T_e\right)\,M^*$.
\end{enumerate}
\end{lemma}

From Lemma \ref{lemma:prop}$(ii)$ we have that the order of multiplication 
in $(i)$ is unimportant.
Lemma~\ref{lemma:prop} and \ref{lemma:repr} were proven in \cite[Lemma~3\&4]{U3}, 
but note that $P_{\rm SB}$ in this article is the lazy version of $P_{\rm SB}$ 
from here.

\section{Proof of Theorem~\ref{th:main}} \label{sec:proof_main}

In this section we will prove Theorem~\ref{th:main}. 
This is done in two subsections. In the first one we prove some 
general norm estimates for operators on (resp. between) Hilbert spaces.
In the second subsection we will apply these estimates to the setting 
from above to obtain the result.

\subsection{Technical lemmas}

In this section we provide some technical lemmas that will be necessary 
for the analysis. We state them in a general form, because we guess that they 
could be useful also in other settings.
First let us introduce the notation. Throughout this section consider two 
Hilbert spaces $H_1$ and $H_2$ with the corresponding inner products 
$\l\cdot,\cdot\r_{H_1}$ and $\l\cdot,\cdot\r_{H_2}$. The norms in ${H_1}$ and 
${H_2}$ are 
defined as usual as the square root of the inner product of a function with 
itself. Additionally, we denote by $\Vert \cdot\Vert_{H_1}$ (resp. 
$\Vert \cdot\Vert_{{H_2}\to {H_1}}$) the operator norms of operators mapping from 
${H_1}$ to ${H_1}$ (resp. ${H_2}$ to ${H_1}$).
We consider two bounded, linear operators, $R$ and $T$. 
The operator $R: {H_2}\to {H_1}$ maps from ${H_2}$ to ${H_1}$ and has the 
\emph{adjoint} $R^*$, i.e. $R^*: {H_1}\to {H_2}$ with 
$\l R^*f,g\r_{H_2}=\l f,R g\r_{H_1}$ for all $f\in {H_1}$ and $g\in {H_2}$. 
The operator $T: {H_2}\to {H_2}$ is self-adjoint and acts on ${H_2}$. 
Obviously, $R T R^*$ is then self-adjoint on ${H_1}$.

\begin{lemma} \label{lemma:tech_mon}
In the setting from above let $T$ be also \emph{positive}, 
i.e. $\l Tg, g\r_{H_2}\ge0$, then
\[
\norm{R T^{k+1} R^*}_{{H_1}} \;\le\; \norm{T}_{{H_2}}\,\norm{R T^k R^*}_{{H_1}}.
\]
\end{lemma}

In the special case $k=1$ this lemma was used in \cite{U3} to prove a 
lower bound on the spectral gap of SW. We will recall this result later.

\begin{proof}
By the assumptions, $T$ has a unique positive square root $\widetilde{T}$, 
i.e. $T=\widetilde{T} \widetilde{T}$, which is again self-adjoint, see e.g. 
\cite[Th. 9.4-2]{Krey}. We obtain
\[\begin{split}
\norm{R T^{k+1} R^*}_{{H_1}}& \;=\; \norm{R \widetilde{T}^{2k+2} R^*}_{{H_1}}
\;=\; \norm{R \widetilde{T}^{k+1}}_{{H_2}\to {H_1}}^2 \\
\;&\le\; \norm{R \widetilde{T}^k}_{{H_2}\to {H_1}}^2 \norm{\widetilde{T}}_{{H_2}}^2 
\;=\; \norm{R \widetilde{T}^{2k} R^*}_{{H_1}}  \norm{T}_{{H_2}} \\
\;&=\; \norm{T}_{{H_2}}\,\norm{R T^k R^*}_{{H_1}}.
\end{split}\]
\end{proof}

In particular, if $\norm{T}_{{H_2}}\le1$ this proves monotonicity in $k$.

\begin{lemma} \label{lemma:tech_in}
In the setting from above let additionally 
$\norm{R}_{{H_2}\to {H_1}}^2 = \norm{R R^*}_{H_1} \le1$, then
\[
\norm{R T R^*}_{{H_1}}^{2^k} \;\le\; \norm{R T^{2^k} R^*}_{{H_1}}
\]
for all $k\in\N$.
\end{lemma}

\begin{proof}
The case $k=0$ is obvious. Now suppose the statement is correct for $k-1$, 
then
\[\begin{split}
\norm{R T R^*}_{{H_1}}^{2^k} \;&=\; \norm{R T R^*}_{{H_1}}^{2^{k-1}\, 2} 
\;\le\; \norm{R T^{2^{k-1}} R^*}_{{H_1}}^2 \\
&\le\; \norm{R T^{2^{k-1}}}_{{H_2}\to {H_1}}^2 \norm{R^*}_{{H_1}\to {H_2}}^2
\;=\; \norm{R T^{2^{k-1}} T^{2^{k-1}} R^*}_{{H_1}} \norm{R R^*}_{{H_2}} \\
&\le\; \norm{R T^{2^k} R^*}_{{H_1}},
\end{split}\]
which proves the statement for $k$. 
\end{proof}

The next corollary combines the statements of the last two lemmas to give 
a result similar to Lemma \ref{lemma:tech_in} for arbitrary exponents.

\begin{corollary} \label{coro:tech}
Additionally to the general assumptions of this section let $T$ be positive, 
$\norm{T}_{{H_2}}\le1$ and $\norm{R R^*}_{{H_1}}\le1$. Then
\[
\norm{R T R^*}_{{H_1}}^{2 k} \;\le\; \norm{R T^k R^*}_{{H_1}}
\]
for all $k\in\N$.
\end{corollary}

\begin{proof}
Let $l=\lfloor\log_2 k \rfloor$ such that $\frac{k}{2}\le2^l\le k$.
Since $\norm{R T R^*}_{{H_1}}\le1$ by assumption, we obtain 
\[
\norm{R T R^*}_{{H_1}}^{2 k} \;\le\; \norm{R T R^*}_{{H_1}}^{2^{l+1}} 
\;\stackrel{L.\text{\scriptsize \ref{lemma:tech_in}}}{\le}\; 
	\norm{R T^{2^{l+1}}  R^*}_{{H_1}}
\;\stackrel{L.\text{\scriptsize \ref{lemma:tech_mon}}}{\le}\; 
	\norm{R T^k  R^*}_{{H_1}}.
\]
\end{proof}

\subsection{Proof}

In this section we apply the estimates from the last one.
Recall that we consider the dynamics on a graph $G=(V,E)$ with $m$ edges, 
i.e. $m=\abs{E}$. Fix an arbitrary ordering $e_1,\dots,e_m$ of the edges 
$e\in E$. 
We set the Hilbert spaces from the last section to 
${H_1}=L_2(\mu)$ and ${H_2}=L_2(\nu)$ 
and define the operators 
\begin{equation*} 
T \;:=\; \frac{1}{m}\,\sum_{i=1}^m\, T_{e_i}
\end{equation*}
and
\begin{equation*} 
\T \;:=\; \prod_{i=1}^m\, T_{e_i}
\end{equation*}
with $T_e$ from \eqref{eq:T-e}. Note that $P_{\rm SB}=M T M^*$ and 
$P_{\rm SW}=M \T M^*$ by Lemma~\ref{lemma:repr}.
Additionally we define
\begin{equation*} 
\T_\a \;:=\; \prod_{i=1}^m\, T_{e_i}^{\a_i}
\end{equation*}
for $\a\in\N^m$. By Lemma~\ref{lemma:prop}(ii) we obtain for $\a, \g\in\N^m$ that 
$\T_\a=\T_\g$ if and only if $\{i:\,\a_i=0\}=\{i:\,\g_i=0\}$. 
Furthermore, $\T_\a=\T$ for every $\a\in\N^m$ with $\a_i>0$ for all 
$i=1,\dots,m$. We prove the following theorem.

\begin{theorem} \label{th:norm}
Let $k = \lceil m \log\frac{m}{\eps}\rceil$ and $R:{H_2}\to {H_1}$ be a bounded, 
linear operator with $\norm{R R^*}_{H_1}\le1$. Then
\[
\norm{R T^k R^*}_{H_1} \;\le\; (1-\eps) \,\bigl\|R \T R^*\bigr\|_{H_1} \,+\, \eps.
\]
\end{theorem}

\begin{proof}
Define 
the index sets $I_{m,k}:=\{\a\in\N^m: \sum_{i=1}^m \a_i =k\}$ and 
$I^1_{m,k}:=\{\a\in I_{m,k}:  \a_i>0, \;\forall i=1,\dots,m\}$. 
Let $I^0_{m,k}:=I_{m,k}\setminus I^1_{m,k}$ and denote by 
$\binom{k}{\a}$, for $\a\in I_{m,k}$, the multinomial coefficient.
Obviously (by the multinomial theorem), 
\[
\sum_{\a\in I_{m,k}} \,\binom{k}{\a} \;=\; m^k
\] 
and 
\[\begin{split}
Z_{m,k} \;:=\; \sum_{\a\in I^0_{m,k}} \,\binom{k}{\a}
\;&\le\; \sum_{i=1}^m\,\sum_{\a\in I^0_{m,k}: \a_i=0} \,\binom{k}{\a}
\;=\; m\,\sum_{\g\in I_{m-1,k}} \,\binom{k}{\g}
\;=\; m(m-1)^k.
\end{split}\] 
We write
\[\begin{split}
T^k \;&=\; \left(\frac1m\sum_{i=1}^m T_{e_i}\right)^k 
\;=\; \frac{1}{m^k} \sum_{\a\in I_{m,k}}\,\binom{k}{\a}\, \T_\a \\
&=\; \frac{1}{m^k} \sum_{\a\in I^1_{m,k}}\,\binom{k}{\a}\, \T_\a 
		\;+\; \frac{1}{m^k} \sum_{\a\in I^0_{m,k}}\,\binom{k}{\a}\, \T_\a.
\end{split}\]
Note that we use for the second equality that the $T_e$'s are commuting by 
Lemma~\ref{lemma:prop}(ii). 
Since we know that $\T_\a=\T$ for every $\a\in I^1_{m,k}$ 
(note that $I^1_{m,k}=\varnothing$ for $k\le m$) and 
$\norm{R \T_\a R^*}_{H_1}\le1$ for every $\a\in I_{m,k}$, we obtain
\[\begin{split}
\norm{R T^k R^*}_{H_1} 
\;&\le\; \frac{1}{m^k} \sum_{\a\in I^1_{m,k}}\,\binom{k}{\a}\, 
						\bigl\|R \T R^*\bigr\|_{H_1}
	\;+\; \frac{1}{m^k} \sum_{\a\in I^0_{m,k}}\,\binom{k}{\a}\, 
						\bigl\|R \T_\a R^*\bigr\|_{H_1} \\
&\le\; \left(1-\frac{Z_{m,k}}{m^k}\right) \,\bigl\|R \T R^*\bigr\|_{H_1}  
	\;+\; \frac{Z_{m,k}}{m^k}.
\end{split}\]
Using $(1-a)c+a\le(1-b)c+b$ for $c\le1$ and $a\le b$, and $\|R \T R^*\|_{H_1}\le1$ 
it follows
\[
\norm{R T^k R^*}_{H_1} \;\le\; 
	\left(1-m\left(1-\frac1m\right)^k\right)\, \bigl\|R \T R^*\bigr\|_{H_1}  
		\;+\; m\left(1-\frac1m\right)^k.
\]
Setting $k = \lceil m \log\frac{m}{\eps}\rceil$ yields the result.
\end{proof}

Now we are able to prove the comparison result for SW and SB dynamics. 
For this let $S_1(B,(\sigma,A)):=\nu(\sigma,A)$ for all 
$B\in\O_{\rm RC}$ and $(\sigma,A)\in\O_{\rm J}$, which defines 
an operator (by \eqref{eq:map}) that maps from $H_2$ to $H_1$. 
The adjoint operator $S_1^*$ is then given by 
$S_1^*((\sigma,A),B):=\mu(B)$ and thus, 
$S_1 S_1^*(A,B) = S_{\mu}(A,B) = \mu(B)$ for all $A,B\in\O_{\rm RC}$. 
For the proof we set
\[
R \;:=\; M - S_1.
\]
It follows that $R R^* = (M-S_1)(M^*-S_1^*) = M M^* - S_\mu$, 
since $M S_1^* = S_1 M^* = S_\mu$, but $M M^*(A,B)=\ind(A=B)$ 
and thus $(R R^*)^2 = (I - S_\mu)^2 = I-S_\mu = R R^*$. 
This implies $\norm{R R^*}_{H_1}=1$. 
Additionally, $P_{\rm SW}-S_\mu = R \T R^*$ and 
$P_{\rm SB}-S_\mu = R T R^*$.

\begin{theorem} \label{th:main-SB}
Let $P_{\rm SW}$ (resp. $P_{\rm SB}$) be the transition matrix of the 
Swendsen-Wang (resp. single-bond) dynamics for the random-cluster model 
on a graph with $m$ edges. Then 
\[
\lambda(P_{\rm SW}) \;\le\; 8 m \log m\; \lambda(P_{\rm SB}).
\]
\end{theorem}

\begin{proof}
Let $k = \lceil m \log\frac{m}{\eps}\rceil$. Then
\[\begin{split}
\lambda(P_{\rm SW}) \;&=\; 1-\norm{R \T R^*}_{H_1} 
\;\stackrel{Th.\text{\scriptsize \ref{th:norm}}}{\le}\; 
		1-\frac{1}{1-\eps}\left(\norm{R T^k R^*}_{H_1}-\eps\right) \\
&=\; \frac{1}{1-\eps}\left(1-\norm{R T^k R^*}_{H_1}\right)
\;\stackrel{Coro.\text{\scriptsize \ref{coro:tech}}}{\le}\; 
		\frac{1}{1-\eps}\left(1-\norm{R T R^*}_{H_1}^{2k}\right) \\
&\le\; \frac{2k}{1-\eps}\left(1-\norm{R T R^*}_{H_1}\right) 
\;=\; \frac{2k}{1-\eps}\,\lambda(P_{\rm SB}), 
\end{split}\]
where the last inequality comes from $1-x^k\le k(1-x)$ for $x\in[0,1]$.
Setting $\eps=\frac12$, we obtian 
$\frac{2k}{1-\eps}=4k\le 8 m \log m$. This proves the statement.\\
\end{proof}

Combining Lemma~\ref{lemma:SB-HB} and Theorem~\ref{th:main-SB} 
proves Theorem~\ref{th:main}.

\section{Proof of Theorems \ref{th:SW_square} and \ref{th:HB_square}} 
\label{sec:proof_square}

In this section we introduce the notion of (planar) dual graphs and prove 
that the heat-bath dynamics on a planar graph $G$ has the same spectral 
gap than the heat-bath dynamics for the dual model on $G^\dag$, 
which is the dual graph of $G$. 
This immediately implies Corollary~\ref{coro:SW_dual} and hence, 
that rapid mixing of the Swendsen-Wang dynamics for the 
random-cluster model and its dual model is equivalent. 
Finally, we use the known lower bounds on the spectral gap 
of SW on the two-dimensional square lattice at high temperatures 
to prove Theorem~\ref{th:SW_square} and Theorem~\ref{th:HB_square}.

\subsection{Dual graphs} \label{subsec:dual}

Let $G$ be a \emph{planar} graph, i.e. a graph that can be embedded 
into a sphere $S^2$ such that two edges of $G$ intersect only at a common endvertex. 
We fix such an embedding for $G$.
Then we define the \emph{dual graph} $G^\dag=(V^\dag,E^\dag)$ of $G$ as follows. 
Place a dual vertex in each face, i.e. in each region of $S^2$ whose boundary consists 
of edges in the embedding of $G$, and connect 2 vertices by the dual edge $e_\dag$ if and 
only if the corresponding faces of $G$ share the boundary edge $e$ 
(see e.g. \cite[Section 8.5]{G2}). 
Note that the dual graph certainly depends on the used embedding.
It is clear, that the number of 
vertices can differ in the dual graph, but we have the same number of 
edges.

Additionally we define a \emph{dual RC configuration} 
$A^\dag\subseteq E^\dag$ in $G^\dag$ to a RC state $A\subseteq E$ in 
$G$ by
\[
e\in A \;\Longleftrightarrow\; e_\dag\notin A^\dag,
\]
where $e_\dag$ is the edge in $E^\dag$ that intersects $e$ in our (fixed) 
embedding. (By construction, this edge is unique.) 

It is easy to obtain (see \cite[p.~134]{G1}) that the random 
cluster models on the (finite) 
graphs $G$ and $G^\dag$ are related by the equality
\begin{equation} \label{eq:dual}
\mu_{p,q}^G(A) \;=\; \mu^{G^\dag}_{p^*,q}(A^\dag),
\end{equation}
where the dual parameter $p^*$ satisfies
\begin{equation} \label{eq:dual-p}
\frac{p^*}{1-p^*} \;=\; \frac{q \,(1-p)}{p}.
\end{equation}
The self-dual point of this relation is given by 
$p_{\rm sd}(q)=\frac{\sqrt{q}}{1+\sqrt{q}}$, which 
corresponds by $p=1-e^{-\beta}$ to the critical temperature 
of the $q$-state Potts model $\beta_c(q)=\log(1+\sqrt{q})$ 
on $\Z^2$ \cite{BDC}. 
For given $G$, $q$ and $p$, we call the RC model on $G^\dag$ 
with parameters $q$ and $p^*$ the \emph{dual model} and, 
we write $\mu^\dag$ instead of $\mu^{G^\dag}_{p^*,q}$, 
if the setting is fixed.

Using this, we obtain 
\[\begin{split}
P_{\rm HB}(A,B) \;&=\; \frac1{2\abs{E}}\,\sum_{e\in E}\,
	\frac{\mu(B)}{\mu(A\cup e)+\mu(A\setminus e)}\;\ind(A\ominus B\subset e) \\
&=\; \frac1{2\abs{E}}\,\sum_{e\in E}\,
	\frac{\mu^\dag(B^\dag)}{\mu^\dag\bigl((A\cup e)^\dag\bigr)+
	\mu^\dag\bigl((A\setminus e)^\dag\bigr)}\;\ind(A\ominus B\subset e) \\
&=\; \frac1{2\abs{E^\dag}}\,\sum_{e_\dag\in E^\dag}\,
	\frac{\mu^\dag(B^\dag)}{\mu^\dag\bigl(A^\dag\setminus e_\dag\bigr)+
	\mu^\dag\bigl(A^\dag\cup e_\dag\bigr)}\;
	\ind(A^\dag\ominus B^\dag\subset e_\dag) \\
&=\; P^\dag_{\rm HB}(A^\dag,B^\dag),
\end{split}\]
where we write $P^\dag_{\rm HB}$ for the heat-bath dynamics for the dual model. 
Since both transition matrices have obviously the same eigenvalues, 
by the above equality, this proves the following lemma.

\begin{lemma} \label{lemma:HB_dual}
Let $P_{\rm HB}$ (resp. $P^\dag_{\rm HB}$) be the transition matrix of 
the heat-bath dynamics for the random-cluster (resp. dual) model. 
Then
\[
\lambda(P_{\rm HB}) \;=\; \lambda(P^\dag_{\rm HB}).
\vspace{5mm}
\]
\end{lemma}


\begin{remark}
Indeed, it is possible to prove a result similar to Lemma~\ref{lemma:HB_dual} 
also for non-planar graphs as long they can be embedded 
(without intersecting edges) in surfaces of 
bounded genus. In this case, the equality becomes an inequality with a constant 
that depends exponentially on the genus. In particular, this leads to the same 
result as in Theorems \ref{th:SW_square} and \ref{th:HB_square} 
for $\Z_L^2$ with periodic boundary condition, i.e. the square lattice on the 
torus. But since this would extend this article by several pages, we decide 
to refer the interested reader to \cite{U-phd}.
\end{remark}

Additionally, we know from \cite[Theorem~5]{U2} that 
\begin{equation} \label{eq:lower}
\lambda\left(\frac{I+P_{\rm SB}}{2}\right) \;\le\; \lambda(P_{\rm SW})
\end{equation}
for arbitrary graphs.
Combining Theorem~\ref{th:main} with Lemma~\ref{lemma:HB_dual}, 
Lemma~\ref{lemma:SB-HB} and \eqref{eq:lower} (in this order) 
proves Corollary~\ref{coro:SW_dual}.

\subsection{Application to the square lattice}

In this subsection we recall the results of \cite{U2} to conclude the proofs 
of Theorems~\ref{th:SW_square} and \ref{th:HB_square}. 
For this let $P_{\rm SW}$ be the Swendsen-Wang dynamics for the 
random-cluster model on $\Z_L^2$ with parameters $p$ and $q$.
First note that, by \cite[Corollary~2]{U2}, we know that there exist constants 
$c_p$, $c'$ and $C$ such that 
\[
\lambda(P_{\rm SW})^{-1} \;\le\; c_p N 
			\qquad \text{ for } p < p_c(q)
\]
and
\[
\lambda(P_{\rm SW})^{-1} \;\le\; c' N^C
			\qquad \text{ for } q=2 \text{ and } p = p_c(2),
\]
where $N:=\abs{V_L}=L^2$. Here, $N\le m\le 2N$ unless $L=1$ ($m:=\abs{E_L}$).

Additionally, it is easy to obtain (using the most natural embedding of $\Z_L^2$ in 
the sphere) that the dual graph $\Z_L^{2*}$ of $\Z_L^2$ is isomorph to 
$\Z_{L-1}^2\cup_{\delta} v^*$, where $v^*$ is an arbitrary vertex in the "outer" 
face of the embedding of $\Z_L^2$ and $\Z_{L-1}^2\cup_{\delta} v^*$ means 
that there is an edge between $v^*$ and every vertex of the \emph{boundary} 
of $\Z_{L-1}^2$. I.e. $\Z_{L-1}^2\cup_{\delta} v^*=(V_L^*,E_L^*)$, where 
$V_L^*=V_{L-1}\cup v^*$ and 
$E_L^*=E_{L-1}\cup\left\{\{v^*,u\}:\,u=(u_1,u_2)\in V_{L-1},\, 
	\{u_1,u_2\}\cap\{1,L-1\}\neq\varnothing\right\}$. 
	
Using Theorem~1$'$ from\cite{U2} (and the discussion after it) we obtain that 
there is a constant $c_p$ such that 
\[
\lambda(P^\dag_{\rm SW})^{-1} \;\le\; c_p N 
			\qquad \text{ for } p^* < p_c(q), 
\]
where $P^\dag_{\rm SW}$ is the Swendsen-Wang dynamics for the RC model on 
$\Z_L^{2*}$ with parameters $p^*$ and $q$, which is the dual model to the RC 
model on $\Z_L^2$ with parameters $p$ and $q$ whenever 
$\frac{p^*}{1-p^*}=\frac{q(1-p)}{p}$. Note that this implies $p^*<p_c$ iff $p>p_c$.
Hence we immediately obtain Theorem~\ref{th:SW_square} from 
Corollary~\ref{coro:SW_dual} and Theorem~\ref{th:HB_square} from 
Theorem~\ref{th:main}.

\section{Conclusion} \label{sec:end}

In this article we gave the first proof of rapid mixing of a Markov chain for 
the Ising (resp. Potts) model on the two-dimensional square lattice at all 
(resp. all non-critical) temperatures. 
Again it should be pointed out that the results of this paper 
(about mixing in $\Z^2$) are ultimately 
due to the numerous results for mixing of the single-spin dynamics for the 
Ising/Potts model, as stated in the introduction, and are therefore far 
from being sharp, because the Swendsen-Wang dynamics is believed to be 
(much) faster than single-spin dynamics.
Here, we mainly considered Markov chains for the closely related 
random-cluster model. 
Surprisingly, there is an upper bound on the spectral gap of the 
Swendsen-Wang dynamics in terms of the spectral gap of the 
heat-bath dynamics that only depends on the number of edges of the graph, 
and does not require any structure of the graph. 
This bound is tight up to the logarithmic factor, e.g. for the RC model on trees.

The proof of this uses some ideas from Markov chain comparison, 
but the main part (of this article) was to find a representation of the Markov 
chains on a larger state space and bounding the norm of the involved 
(Markov) operators. We guess that this technique can be used also in 
other settings, such as Markov chains in general state spaces, to 
compare Markov chains with small and large "step sizes".

As stated above, the proof is via comparison and it would still be interesting 
to find a direct proof of rapid mixing of the SW dynamics, especially 
at the critical temperature.


{
\linespread{1}
\bibliographystyle{amsalpha}
\bibliography{Bibliography}
}

\end{document}